\documentclass[12pt,reqno]{amsart}

\usepackage{graphicx}
\usepackage{epstopdf}
\usepackage{psfrag}

\usepackage{amsfonts,amsmath,amssymb, amsthm}
\usepackage[usenames]{color}
\newtheorem{theorem}{Theorem}[section]
\newtheorem{proposition}[theorem]{Proposition}
\newtheorem{lemma}[theorem]{Lemma}
\newtheorem{corollary}[theorem]{Corollary}
\newtheorem{remark}[theorem]{Remark}
\newtheorem{definition}[theorem]{Definition}

\newcommand{\be}{\begin{equation}}
\newcommand{\ee}{\end{equation}}

\numberwithin{equation}{section}

\newcommand{\Ch}{{\rm Ch}}
\newcommand{\sfd}{{\sf d}}

\DeclareMathOperator{\supp}{spt}

\newcommand{\RCD}{\mathsf{RCD}}
\newcommand{\CD}{\mathsf{CD}}
\newcommand{\mm}{\mathfrak m}
\newcommand{\gopt}{{\rm{OptGeo}}}
\newcommand{\Tan}{{\rm Tan}}

\newcommand{\ppi}{{\mbox{\boldmath$\pi$}}}

\newcommand{\geo}{{\rm{Geo}}}
\newcommand{\e}{{\rm{e}}}

\newcommand{\R}{\mathbb{R}}

\renewcommand{\dim}{{\textrm{dim}}}

\renewcommand{\d}{{\textrm {d}}}
\newcommand{\Lip}{{\textrm {Lip}}}

\newcommand\N{{\mathbb N}}

\newcommand\cH{{\mathcal H}}

\begin{document}

\title[On the volume measure of $\RCD^{*}(K,N)$-spaces]{On the volume measure of  non-smooth spaces with Ricci curvature bounded below}

\author {Martin Kell} \thanks{M. Kell, Universit\"at T\"ubingen,  Fachbereich Mathematik, email: \textsf{martin.kell@math.uni-tuebingen.de}}   
\author{Andrea Mondino}  \thanks{A. Mondino,  Warwick University, Department of Mathematics, email: \textsf{A.Mondino@warwick.ac.uk}}

\maketitle

\begin{abstract} 
We prove that, given an $\RCD^{*}(K,N)$-space $(X,\sfd,\mm)$, then it is possible to  $\mm$-essentially cover $X$ by measurable subsets $(R_{i})_{i\in \N}$ with the following property: for each $i$  there exists $k_{i} \in \N\cap [1,N]$ such that $\mm\llcorner R_{i}$ is absolutely continuous with respect to the $k_{i}$-dimensional Hausdorff measure.   We  also show that a Lipschitz differentiability space which is locally  bi-Lipschitz embeddable into euclidean spaces is rectifiable as a metric measure space, and we conclude with  an application to Alexandrov spaces.
  \end{abstract}

\tableofcontents

\section{Introduction}
The object of this  note  is to investigate the volume measure of non smooth spaces having Ricci curvature bounded from below in a synthetic sense, the so called $\RCD^{*}(K,N)$-spaces where $K\in \R$ stands for the lower bound on the Ricci curvature and $N\in (1,\infty)$ stands for an upper bound on the dimension.  More precisely these non-smooth objects  are metric measure spaces, i.e. triples $(X,\sfd,\mm)$ where $(X,\sfd)$ is a complete and separable metric space endowed with a Borel positive measure $\mm$, such that the optimal transportation satisfies suitable convexity properties and the Sobolev space $W^{1,2}$ is a Hilbert space (see Section \ref{SS:RCD} for the precise notions). 
Let us mention that $\RCD^{*}(K,N)$-spaces can be seen as the \emph{Ricci-curvature} counterpart of Alexandrov spaces, which are metric spaces satisfying  \emph{sectional-curvature} lower bounds in a synthetic sense.  Note that while sectional curvature lower bounds involve only the distance function, so they make perfect sense for a \emph{metric space}, on the other hand Ricci curvature lower bounds involve the interplay of distance and volume so they make sense for \emph{metric measure spaces}.

Two key properties of the $\RCD^{*}(K,N)$-condition are the compatibility with the smooth counterpart (i.e. a Riemannian manifold endowed with the Riemannian distance and volume measure is an $\RCD^{*}(K,N)$-space if and only if it has Ricci $\geq K$ and dimension $\leq N$) and the stability with respect to measured Gromov-Hausdorff convergence, mGH for short. The combination of these two properties clearly implies that Ricci-limit spaces (i.e. mGH-limits of Riemannian manifolds having Ricci $\geq K$ and dimension $\leq N$) are $\RCD^{*}(K,N)$-spaces. Notice that the class of $\RCD^{*}(K,N)$-spaces contains both \emph{collapsed} and \emph{non-collapsed}  Ricci-limits, but a priori may be larger.  Let us mention that the corresponding question is  an  open problem even in the  classical framework of Alexandrov spaces (i.e. it is not known if every $n$-dimensional Alexandrov space with curvature bounded below by $k$ is a GH-limit of a sequence of Riemannian manifolds with Sectional curvatures $\geq k$ and dimensions $\geq n$).

The structure of Ricci-limits was deeply investigated by J. Cheeger and T. H. Colding \cite{Cheeger-Colding97I}, \cite{Cheeger-Colding97II}, \cite{Cheeger-Colding97III}. They proved, among other fundamental results, that Ricci-limits  are rectifiable as metric measure spaces, i.e. if $(X,\sfd,\mm)$ is a Ricci-limit then it is possible to cover it (up a subset of $\mm$-measure zero) by charts $R_{j}$ which are bi-Lipschitz onto subsets of $\R^{k_{j}}$, the dimension $k_{j}$ possibly depending on the chart itself,  and such that $\mm\llcorner R_{j}$ is absolutely continuous with respect to the $k_{j}$-dimensional measure $\cH^{k_{j}}$. It was later proved by T. H. Colding and A. Naber \cite{CoNa} that the dimension is independent of the chart.

It is then a natural question if the same statements are true for $\RCD^{*}(K,N)$-spaces.  In a joint work of A. Naber  with the second named author \cite{MN},  it was proved that $\RCD^{*}(K,N)$-spaces are rectifiable as   \emph{metric spaces}; more precisely the following statement holds. Before stating it let us recall that the  \emph{$k$-dimensional regular set}  ${\mathcal R}_{k}$ is   made of those points having unique blow up and such a unique blow-up is euclidean of dimension $k$. More explicitly, $\bar{x}\in {\mathcal R}_{k}$ if  and only  if the sequence of  rescaled spaces $(X, r^{-1} \sfd, \mm^{\bar{x}}_{r}, \bar{x})$ converges in pointed measured Gromov Hausdorff sense to $(\R^{k}, \sfd_{\R^{k}}, c_{j} \cH^{k}, 0^{k})$ as $r \to 0^{+}$, where $c_{k}$ is the renormalization constant such that $\int_{B_{1}(0^{k})} (1-|x|) \, d (c_{k} \cH^{k})=1$ and 
$$\mm^{\bar{x}}_{r}:= \left(\int_{B_{r}(\bar{x})}  \left( 1- \frac{\sfd(\bar{x},\cdot)}{r} \right) \, d \mm \right)^{-1} \mm.$$
From Bishop-Gromov inequality, it is easily seen that ${\mathcal R}_{k}=\emptyset$ for all $k\geq \lfloor N \rfloor$, where $\lfloor N \rfloor$ denotes the integer part of $N$.

\begin{theorem}[Rectifiability of  $\RCD^{*}(K,N)$-spaces,  \cite{MN}]\label{thm:rectIntro}
Let $(X,\sfd,\mm)$ be an $\RCD^{*}(K,N)$-space for some $K\in \R, N \in (1,\infty)$. Then $\mm(X\setminus \bigcup_{k=1}^{\lfloor N \rfloor} {\mathcal R}_{k})=0$ and   every  $k$-dimensional regular set  ${\mathcal R}_{k}$ is $k$-rectifiable, i.e. it can be covered up to an $\mm$-negligible set by Borel subsets which are bi-Lipschitz to Borel subsets of $\R^{k}$.  In particular $(X,\sfd,\mm)$ is rectifiable as a \emph{metric space}.
\end{theorem}

One may wonder whether, more strongly,  $\RCD^{*}(K,N)$-spaces are rectifiable as \emph{metric measure spaces}. More precisely whether or not $\mm\llcorner {\mathcal R}_{k}$ is absolutely continuous with respect to   $\cH^{k}$. The main goal of the present paper is to answer affirmatively to such a question, that is to  prove the next result.

\begin{theorem}\label{thm:main}
Let $(X,\sfd,\mm)$ be an $\RCD^{*}(K,N)$-space for some $K\in \R, N \in (1,\infty)$ and let  ${\mathcal R}_{k}$ be  the  $k$-dimensional regular set.   Then $\mm\llcorner {\mathcal R_{k}}$ is absolutely continuous with respect to $\cH^{k}$. As a consequence,  $(X,\sfd,\mm)$ is rectifiable as  a \emph{metric measure space}.
\end{theorem}

Let us mention that the proof for Ricci-limits in \cite{Cheeger-Colding97III} was performed by getting estimates on the smooth approximating sequence and then pass to the limit; since we do not have at disposal smooth approximations of an $\RCD^{*}(K,N)$-space, our strategy is completely different and we work directly on the non-smooth space $X$ itself. More precisely, we first observe that an $\RCD^{*}(K,N)$-space is a Lipschitz differentiability space (see Section \ref{SS:LDS} for the definition and some basic properties), indeed since an $\RCD^{*}(K,N)$-space is locally doubling and satisfies a local Poincar\'e inequality (as resp. proved  by K. Bacher and K. T. Sturm \cite{BS2010}, and by T. Rajala \cite{Rajala}) the claim follows by the celebrated work of  J. Cheeger \cite{Cheeger}. In particular, since by Theorem  \ref{thm:rectIntro} we can locally embed $X$ into a euclidean space, it is clear the link with the Lipschitz differentiability spaces which are embeddable in euclidean spaces.
It was proved by  J. Cheeger \cite{Cheeger} (for embeddable PI spaces) and   G.C. David \cite{David}   that  such a space has almost everywhere unique tangent which  is moreover isometric to a euclidean space having the same dimension of the Lipschitz chart (in the sense of the Lipschitz differentiable structure).  In a rather explicit and independent way, in Section \ref{S:RadonVsLDS} we show that, more strongly,   a Lipschitz differentiability space locally  embeddable  in euclidean spaces is rectifiable as  metric measure space. Recall  that a m.m.s. $(X,\sfd,\mm)$ is \emph{locally bi-Lipschitz embeddable in euclidean spaces} if there exist Borel subsets $\{E_{j} \subset X \}_{j \in \N}$ such that each $E_{j}$ is bi-Lipschitz embeddable in some euclidean space $\R^{N_{j}}$ and  $\mm(X\setminus \bigcup_{j \in \N} E_{j})=0$. 

\begin{theorem} [Theorem \ref{cor:tangLDS}] \label{thm:main2}
Let  $(X,\sfd,\mm)$ be a Lipschitz differentiability space and assume it is locally bi-Lipschitz embeddable in euclidean spaces. Then there exists a countable collection $\{R_{j}\}_{j \in \N}$ of Borel subsets of $X$, covering $X$ up  to an $\mm$-negligible set, such that each $R_{j}$ is bi-Lipschitz to a Borel subset of $\R^{n_j}$ and $\mm \llcorner R_j$ is absolutely continuous with respect to the $n_j$-dimensional Hausdorff measure $\cH^{k_{j}}$.  In other words, $(X,\sfd,\mm)$ is rectifiable as a \emph{metric measure space}.
\end{theorem}

We briefly mention that in order to build the bi-Lipschitz charts we make use of a construction of P. Mattila  \cite{Mattila95}, and in order to control the measure we combine the work  of G. Alberti and A. Marchese \cite{AM} with the recent paper of G. De Philippis and F. Rindler \cite{DePRin} on the structure of Radon measures in euclidean spaces, and the work of D. Bate \cite{Bate} on Alberti representations of Lipschitz differentiability spaces.

We then  show that  the dimension of the chart in the Lipschitz differentiable structure of an $\RCD^{*}(K,N)$-space $X$ is the same as the dimension of the bi-Lipschitz chart in the rectifiability Theorem  \ref{thm:rectIntro} (see Proposition \ref{prop:ni=kj} and compare with \cite{CR}). At this point, Theorem \ref{thm:main} will follow from Theorem \ref{thm:main2}.
\medskip

We conclude with a couple of applications, the first one to Alexandrov geometry; to this aim let us recall that if $(X,\sfd)$ is an $n$-dimensional Alexandrov space with curvature bounded from below, then there exists an open dense set ${\mathcal R}\subset X$ with $\cH^{n}(X \setminus {\mathcal R})=0$  such that every its point has an open neighborhood bi-Lipschitz homeomorphic to an open region of $\R^{n}$ (see for instance \cite[Theorem 10.8.3 and Section 10.10.3]{BBI}).

\begin{corollary}
Let $(X,\sfd)$ be an $n$-dimensional Alexandrov space with curvature bounded from below, $n\in \N, n\geq 2$. Let  $\mm$ be  a positive Radon measure over $(X,\sfd)$ with $\supp \mm=X$, $\mm(X\setminus {\mathcal R})=0$ and such that the metric measure space $(X,\sfd,\mm)$ is an $\RCD^{*}(K,N)$-space for some $K\in \R$, $N\in (1,\infty)$ (or, more generally, $(X,\sfd,\mm)$ is a Lipschitz differentiability space). 

Then $\mm$ is absolutely continuous with respect to $\cH^{n}$. 
\end{corollary}

Let us mention that related questions and results about the structure of measures satisfying Ricci curvature lower bounds have been studied by F. Cavalletti and the second author in \cite{CM}, for the specific case of Alexandrov space see \cite[Corollary 7.4]{CM}. As a final application we wish to investigate the uniqueness of measures on an $\RCD^{*}(K,N)$-space such that the resulting space is still $\RCD$ or, more generally is a Lipschitz differentiability space.
In order to state the result precisely, let us call ${\mathcal R}\subset X$ be any maximal set (with respect to inclusion)  which can be covered by Borel subsets  of $X$ bi-Lipschitz to Borel sets in euclidean spaces; by  Theorem \ref{thm:rectIntro} we know that $\mm(X \setminus {\mathcal R})=0$.

\begin{corollary}
Let $(X,\sfd,\mm)$ be an $\RCD^{*}(K,N)$-space for some $K\in \R, N \in (1,\infty)$ and let  ${\mathcal R}$ be any set  as above.

Let $\tilde{\mm}$ be another  positive Radon measure over $(X,\sfd)$ with $\supp \tilde{\mm}=X$, $\tilde{\mm}(X\setminus {\mathcal R})=0$ and such that the metric measure space $(X,\sfd, \tilde{\mm})$ is an $\RCD^{*}(\tilde{K},\tilde{N})$-space for some  $\tilde{K}\in \R, \tilde{N} \in (1,\infty)$ (or more generally $(X,\sfd, \tilde{\mm})$ is a Lipschitz differentiability space). Then $\tilde{\mm} \llcorner {\mathcal R}_{k}$ is absolutely continuous with respect to $\cH^{k}$. 
\end{corollary}

In other words, every $\RCD$ measure is forced to be in the class of measures  which are absolutely continuous with respect to the relevant Hausdorff measure of each stratum ${\mathcal R}_{k}$. 

The proof of the last two corollaries is omitted as it can be performed by  following verbatim the proof of Theorem \ref{thm:main}.
\\

After finishing this note we got to know that with a similar approach,  a related result has been proven by De Philippis, Marchese and Rindler \cite{DePMarRin}; nevertheless it is worth to point out some crucial differences. In  \cite{DePMarRin} it is proved that if $\varphi:U\to \R^{n}$, $U\subset X$ Borel subset, is a Lipschitz chart in a Lipschitz differentiability space $(X,\sfd,\mm)$ then $\varphi_{\sharp}\mm \ll \cH^{n}$; this answers positively to a conjecture of Cheeger \cite{Cheeger}. As the example of the Heisenberg group shows, in general it is not possible to conclude that $\mm \llcorner U \ll \cH^{n}$; so one cannot directly deduce our Theorem  \ref{thm:main} from \cite{DePMarRin}.
Indeed in our results, a crucial additional information (guaranteed  by Theorem \ref{thm:rectIntro}) is that the  space is locally bi-Lipschitz embeddable in  euclidean spaces, an assumption which is used throughout the paper.  Another difference from the technical point of view is that  we reduce the arguments to directly  apply the  weak converse Rademacher Theorem  \ref{thm:DFR} proved in \cite{DePRin}, while in \cite{DePMarRin} the absolute continuity of the measure is achieved by a slightly different argument involving one-dimensional currents.

Let us finally mention that a couple of weeks after our preprint appeared, Gigli and Pasqualetto \cite{GiPa} posted an independent paper proving Theorem  \ref{thm:main} with a different approach,   still relying on the 1-dimensional currents formulation of the  weak converse of Rademacher Theorem proved in \cite{DePRin}  but more analytic in nature: basically they prove that the bi-Lipschitz charts of Theorem \ref{thm:rectIntro} induce vector fields with measure-valued divergence (by using the  Laplacian Comparison Theorem), and therefore give independent $1$-dimensional currents. 

\subsection*{Acknowledgments} The authors wish to thank the anonymous referee for the careful reading and the valuable comments which improved the exposition of the paper. 

\section{Preliminaries}
\subsection{Lipschitz differentiability spaces} \label{SS:LDS} 
Throughout  the paper $(X,\sfd)$ will denote a metric space; most of  the times it will be assumed  complete, but in general it may not. A metric measure space is a triple $(X,\sfd,\mm)$, where $\mm$ is a  Borel positive measure defined over the metric space $(X,\sfd).$ When the metric (and measure) are understood from the context, we will denote such a space simply
by $X$. In order to emphasize the dependency of the metric $\sfd$ on $X$ we may also write $\sfd_{X}$. We denote open and closed balls in $(X,\sfd)$ of center $x$ and radius $r>0$ respectively by $B_{r}(x)$ and $\bar{B}_{r}(x)$.
\\Given a real valued function $f:X \to \R$, its \emph{local Lipschitz constant} at $x_{0} \in X$ is  denoted by $\Lip f(x_{0})$ and defined by 
$$
\Lip f(x_{0}):=\limsup_{x\to x_{0}} \frac{|f(x)-f(x_{0})|}{\sfd(x,x_{0})} \text{ if $x_{0}$ is not isolated,    $\quad \Lip f(x_{0})=0$ otherwise}. 
$$
Recall that if $(X,\sfd_{X})$ and $(Y,\sfd_{Y})$ are metric spaces, then a mapping $f:X\to Y$ is Lipschitz if there is a constant
$L>0$ such that
\begin{equation}\label{eq:LipMap}
\sfd_{Y}(f(x_{1}), f(x_{2})) \leq L \; \sfd_{X} (x_{1}, x_{2}),
\end{equation}
for any two points $x_{1}, x_{2} \in X$. The infimal value of $L$ such that equation  \eqref{eq:LipMap} holds is called \emph{Lipschitz constant} of $f$.
The mapping $f$  is called \emph{bi-Lipschitz} if there is a constant $L\geq 1$ such that
\begin{equation}
L^{-1} \;   \sfd_{X} (x_{1}, x_{2})\leq \sfd_{Y}(f(x_{1}), f(x_{2})) \leq L \; \sfd_{X} (x_{1}, x_{2}),
\end{equation}
for any two points $x_{1}, x_{2} \in X$. Two spaces are said to be bi-Lipschitz equivalent if there is a bi-Lipschitz map of
one onto the other.
A  non-trivial Borel regular measure $\mm$ over the metric space $(X,\sfd)$ is said to be \emph{doubling} if  there exists
a constant  $C>0$ such that
\begin{equation}\label{eq:doubling}
\mm(B_{2r}(x)) \leq C\; \mm(B_{r}(x))
\end{equation}
for all $x\in X$ and $r>0$. We say that $\mm$ is \emph{locally doubling} if for every bounded  subset $K\subset X$, the estimate \eqref{eq:doubling} holds for every $x \in K$ for some $C>0$ possibly depending on $K$.  If $\mm$ is a (resp. locally) doubling measure on the metric space $(X,\sfd)$, then $(X, \sfd)$ is a (resp. locally) doubling metric space, which means
that there exists a constant $N>0$, depending only on the doubling constant associated to $\mm$ (resp. and on the compact subset $K\subset X$), such that every
ball of radius $2r$ in $X$ (resp. centered at a point $x \in K$) can be covered by at most $N$ balls of radius $r$. 
We write $\cH^{k}$ for the  $k$-dimensional Hausdorff measure.

\begin{definition}\label{def:LDS}
A \emph{Lipschitz differentiability space} is a (possibly non-complete) metric measure space $(X,\sfd,\mm)$ satisfying the following
condition: There are positive measure sets (called \emph{charts}) $U_{i}$ covering $X$, positive integers $n_{i}$ (called  \emph{dimensions of the
charts}), and Lipschitz maps $\phi_{i}:U_{i} \to \R^{n_{i}}$ with respect to which any Lipschitz function is differentiable almost
everywhere, in the sense that for $\mm$-almost every $x\in U_{i}$, there exists a unique $df_{x}\in \R^{n_{i}}$ such that
\begin{equation}\label{eq:Differential}
\lim_{y\to x} \frac{|f(y)-f(x)-df_{x}\cdot \big(\phi_{i}(y)-\phi_{i}(x)\big)| }{\sfd(x,y)}=0.
\end{equation}
Here $df_{x}\cdot \big(\phi_{i}(y)-\phi_{i}(x)\big)$ denotes the standard scalar product between elements of $\R^{n_{i}}$.
\end{definition}

If a reference point $\bar{x}\in X$ is fixed, we will call the triple $(X,\sfd,\bar{x})$ a \emph{pointed} metric space. We now recall the notion of tangent space to a pointed metric space. 
\begin{definition}
A  pointed  metric space $(Y,\sfd_{Y}, \bar{y})$ is said to be a \emph{tangent space} to $X$ at $\bar{x}$ if there exists a sequence $r_{i}\downarrow 0$ such that the rescaled pointed spaces $(X, r_{i}^{-1} \sfd, \bar{x})$ converge to 
$(Y,\sfd_{Y}, \bar{y})$ in the pointed Gromov-Hausdorff topology.  The collection of all tangent spaces to $X$ at $\bar{x}$ is denoted by $Tan(X,\bar{x})$.
\end{definition}
By Gromov's Compactness Theorem it follows that if $(X,\sfd)$ is locally doubling then $Tan(X,\bar{x})$ is not empty. A delicate issue is instead the uniqueness of tangent spaces, which clearly in the general framework of locally doubling spaces fails but in many interesting geometric situations holds true, as we will discuss in the next sections.

Let us also recall the following localization result of Lipschitz differentiability spaces, for the proof see for instance \cite[Corollary 4.6]{Bate} or \cite[Corollary 2.7]{BateS}.

\begin{proposition}\label{prop:LocLDS}
Let $(X,\sfd,\mm)$ be a Lipschitz differentiability space and let $U\subset X$ be a subset with $\mm(U)>0$. Then $(U, \sfd|_{U\times U}, \mm \llcorner U)$ is itself a  Lipschitz differentiability space with respect to the same charts structure.
\end{proposition}

\subsection{$\RCD^{*}(K,N)$-spaces} \label{SS:RCD}

In this section  we quickly recall the definition of $\RCD^{*}(K,N)$-space and those properties used in the paper; for brevity we will not mention many interesting results in the field that will not be needed in the present work.

Throughout the section $(X,\sfd)$ will be a complete and separable metric space endowed with a positive Borel measure $\mm$   finite on bounded subsets.
Recall that a curve $\gamma:[0,1]\to X$ is a \emph{geodesic} if
\begin{equation}
\sfd(\gamma(s),\gamma(t))=|s-t|\; \sfd(\gamma(0), \gamma(1)).
\end{equation}
We denote  by $\geo(X)$ the space of geodesics on $(X,\sfd)$ endowed with the $\sup$ distance, and by $\e_t:\geo(X)\to X$, $t\in[0,1]$, the evaluation maps defined by $\e_t(\gamma):=\gamma_t$.
$(X,\sfd)$ is a \emph{geodesic space} if every couple of points of $X$ are joined by a geodesic.

We denote by ${\mathcal  P}(X)$ the space of Borel probability measures on $(X,\sfd)$ and by ${\mathcal  P}_{2}(X) \subset {\mathcal  P}(X)$ the subspace consisting of all the probability measures with finite second moment. For $\mu_0,\mu_1 \in {\mathcal  P}_{2}(X)$ the quadratic transportation distance $W_2(\mu_0,\mu_1)$ is defined by
\begin{equation}\label{eq:Wdef}
  W_2^2(\mu_0,\mu_1) = \inf_\pi \int_X \sfd^2(x,y) \,\d\pi(x,y),
\end{equation}
where the infimum is taken over all $\pi \in {\mathcal  P}(X \times X)$ with $\mu_0$ and $\mu_1$ as the first and the second marginal.

Assuming the space $(X,\sfd)$ to be geodesic, the space $({\mathcal  P}_{2}(X), W_2)$ is also geodesic.  It turns out that any geodesic $(\mu_t) \in \geo({\mathcal  P}_{2}(X))$ can be lifted to a measure $\ppi \in {\mathcal  P}(\geo(X))$, so that $(\e_t)_\sharp \ppi = \mu_t$ for all $t \in [0,1]$. Given $\mu_0,\mu_1\in{\mathcal  P}_{2}(X)$, we denote by $\gopt(\mu_0,\mu_1)$ the space of all
$\ppi \in {\mathcal  P}({\geo(X)})$ for which $(\e_0,\e_1)_\sharp \ppi$ realizes the minimum in \eqref{eq:Wdef}. If $(X,\sfd)$ is geodesic, then the set $\gopt(\mu_0,\mu_1)$ is non-empty for any $\mu_0,\mu_1\in{\mathcal  P}_{2}(X)$.

We turn to the formulation of the $\CD^*(K,N)$ condition, coming from  \cite{BS2010}.  Given $K \in \R$ and $N \in [1, \infty)$, we define the distortion coefficient $[0,1]\times\R^+\ni (t,\theta)\mapsto \sigma^{(t)}_{K,N}(\theta)$ as
\[
\sigma^{(t)}_{K,N}(\theta):=\left\{
\begin{array}{ll}
+\infty,&\qquad\textrm{ if }K\theta^2\geq N\pi^2,\\
\frac{\sin(t\theta\sqrt{K/N})}{\sin(\theta\sqrt{K/N})}&\qquad\textrm{ if }0<K\theta^2 <N\pi^2,\\
t&\qquad\textrm{ if }K\theta^2=0,\\
\frac{\sinh(t\theta\sqrt{K/N})}{\sinh(\theta\sqrt{K/N})}&\qquad\textrm{ if }K\theta^2 <0.
\end{array}
\right.
\]
\begin{definition}[Curvature dimension bounds]
Let $K \in \R$ and $ N\in[1,  \infty)$. We say that a m.m.s.  $(X,\sfd,\mm)$
 is a $\CD^*(K,N)$-space if for any two measures $\mu_0, \mu_1 \in {\mathcal  P}(X)$ with support  bounded and contained in $\supp \mm $ there
exists a measure $\ppi \in \gopt(\mu_0,\mu_1)$ such that for every $t \in [0,1]$
and $N' \geq  N$ we have
\begin{equation}\label{eq:CD-def}
-\int\rho_t^{1-\frac1{N'}}\,\d\mm\leq - \int \sigma^{(1-t)}_{K,N'}(\sfd(\gamma_0,\gamma_1))\rho_0^{-\frac1{N'}}+\sigma^{(t)}_{K,N'}(\sfd(\gamma_0,\gamma_1))\rho_1^{-\frac1{N'}}\,\d\ppi(\gamma), 
\end{equation}
where for any $t\in[0,1]$ we  have written $(\e_t)_\sharp\ppi=\rho_t\mm+\mu_t^s$  with $\mu_t^s \perp \mm$.
\end{definition}
Notice that if $(X,\sfd,\mm)$ is a $\CD^*(K,N)$-space, then so is $(\supp \mm,\sfd,\mm)$, hence it is not restrictive to assume that $\supp \mm=X$, a hypothesis that we shall always implicitly do from now on. 

On $\CD^*(K,N)$-spaces a natural version of the Bishop-Gromov volume growth estimate holds (see \cite[Theorem 6.2]{BS2010}), in particular a $\CD^{*}(K,N)$-space is locally doubling. 
Moreover, as proved by T. Rajala \cite{Rajala}, $\CD^{*}(K,N)$-spaces satisfy a local Poincar\'e inequality. Combining the local doubling and the Poincar\'e inequality with the celebrated work of J. Cheeger \cite{Cheeger}
 we get the following result.
 
 \begin{theorem}\label{thm:RCDLDS}
 Every $\CD^{*}(K,N)$-space is a Lipschitz differentiability space. 
 \end{theorem}

One crucial property of the $\CD^*(K,N)$ condition is the stability under measured Gromov-Hausdorff convergence of m.m.s., so that Ricci limit spaces are $\CD^{*}(K,N)$. Moreover, on the one hand Finsler manifolds are allowed as $\CD^{*}(K,N)$-space while on the other hand from the work of Cheeger-Colding \cite{Cheeger-Colding97I},\cite{Cheeger-Colding97II},\cite{Cheeger-Colding97III}  it was understood that purely Finsler structures never appear as Ricci limit spaces.   Inspired by this fact, in \cite{Ambrosio-Gigli-Savare11b}, Ambrosio-Gigli-Savar\'e proposed a  strengthening of the $\CD$-condition in order to enforce, in some weak sense, a Riemannian-like behavior of spaces with a curvature-dimension bound (to be precise in  \cite{Ambrosio-Gigli-Savare11b} it was analyzed the case of strong-$\CD(K,\infty)$ spaces endowed with a probability reference measure $\mm$; the axiomatization has been then simplified and generalized in  \cite{AmbrosioGigliMondinoRajala} to allow $\CD(K,\infty)$-spaces endowed with a $\sigma$-finite reference measure); the finite dimensional refinement led to the  $\RCD^{*}(K,N)$ condition. 

Such a strengthening consists in requiring that the space $(X,\sfd,\mm)$ is such that the Sobolev space $W^{1,2}(X,\sfd,\mm)$ is Hilbert, a condition we shall refer to as `infinitesimal Hilbertianity'. It is out of the scope of this note to provide full details about the definition of $W^{1,2}(X,\sfd,\mm)$ and its relevance in connection with Ricci curvature lower bounds. We will instead be satisfied by recalling the definition and a structural result proved by A. Naber and the second named author \cite{MN} which will play a key role in the present paper.

First of all recall that on a m.m.s. there is a canonical notion of `modulus of  the differential of a function' $f$, called weak upper differential and denoted with $|Df|_w$; with this object one defines the Cheeger energy
$$\Ch(f):=\frac 1 2 \int_X |Df|_w^2 \, \d \mm.$$
The  Sobolev space $W^{1,2}(X,\sfd,\mm)$ is by definition  the space of $L^2(X,\mm)$ functions having finite Cheeger energy, and it is endowed with the natural norm  $\|f\|^2_{W^{1,2}}:=\|f\|^2_{L^2}+2 \Ch(f)$ which makes it a Banach space. We remark that, in general, $W^{1,2}(X,\sfd,\mm)$ is not Hilbert (for instance, on a smooth Finsler manifold the space $W^{1,2}$ is Hilbert if and only if the manifold is actually Riemannian); in case  $W^{1,2}(X,\sfd,\mm)$ is  Hilbert then we say that $(X,\sfd,\mm)$ is \emph{infinitesimally Hilbertian}. 

\begin{definition}
An $\RCD^{*}(K,N)$-space $(X,\sfd,\mm)$ is an infinitesimally Hilbertian $\CD^*(K,N)$-space.
\end{definition}

The following structural result, proved by A. Naber and the second named author \cite{MN},   will play a key role in the present paper.

\begin{theorem}[Rectifiability of of $\RCD^{*}(K,N)$-spaces]  \label{RCD-rect}
Let $(X,\sfd,\mm)$ be an $\RCD^{*}(K,N)$-space for some $K\in \R, N \in (1,\infty)$. Then there exists a countable collection $\{R_{j}\}_{j \in \N}$ of $\mm$-measurable subsets of $X$, covering $X$ up to an $\mm$-negligible set, such that each $R_{j}$ is bi-Lipschitz to a measurable subset of $\R^{k_{j}}$, for some $k_{j}\in \N$, possibly depending on $j$. Moreover for $\mm$-a.e. $x \in R_{j}$ the tangent space is unique and isometric to $\R^{k_{j}}$.
\end{theorem}

\section{Structure of Radon measures in Euclidean spaces Vs Lipschitz differentiability spaces} \label{S:RadonVsLDS}

The goal of this section is to investigate the structure of Lipschitz differentiability spaces which can be locally bi-Lipschitz embedded in euclidean spaces. Recall  that a m.m.s. $(X,\sfd,\mm)$ is \emph{locally bi-Lipschitz embeddable in euclidean spaces} if there exist Borel subsets $\{E_{j} \subset X \}_{j \in \N}$ such that each $E_{j}$ is bi-Lipschitz embeddable in some euclidean space $\R^{N_{j}}$ and  $\mm(X\setminus \bigcup_{j \in \N} E_{j})=0$. Notice that by Proposition \ref{prop:LocLDS}, each $E_{j}$ endowed with the induced metric measure structure is a Lipschitz differentiability space which is globally bi-Lipschitz embeddable in $\R^{N_{j}}$.
    It is known from the works of J. Cheeger (for  PI-spaces globally bi-Lipschitz embeddable in some $\R^{N}$ \cite[Theorem 14.1, 14.2]{Cheeger}) and of  G.C. David  (for complete doubling Lipschitz differentiability spaces embedded in $\R^{N}$ \cite[Corollary 8.1]{David}) that such spaces have a.e. a unique tangent space which is isometric to a euclidean space of the same dimension of the  Lipschitz chart.

In the present section we prove that, more strongly,  a Lipschitz differentiability space locally bi-Lipschitz embeddable  in  euclidean spaces is rectifiable as a metric measure space. Our proof is independent on the ones mentioned above and we directly construct the rectifiability charts using the Lipschitz differentiability. We start by showing that such spaces are rectifiable as \emph{metric spaces}.

\begin{theorem}\label{thm:tangLDS}
Let  $(X,\sfd,\mm)$ be a Lipschitz differentiability space, let $(U,\phi)$ be an $n$-dimensional chart and assume that $(X,\sfd,\mm)$ is locally    bi-Lipschitz embeddable in euclidean spaces.  Then there exists a countable collection $\{R_{j}\}_{j \in \N}$ of $\mm$-measurable subsets of $X$, covering $U$ up to an $\mm$-negligible set, such that each $R_{j}$ is bi-Lipschitz to a measurable subset of $\R^{n}$. In other words $U$ is $n$-rectifiable as a \emph{metric space}.
\end{theorem}

For $w\in\mathbb{S}^{N-1}$ and $\theta\in[0,\frac{\pi}{2}]$
the (closed) cone of width $\theta$ centered at $w$ is the set 
\[
C(w,\theta)=\{v\in\mathbb{R}^{N}\,|\, v\cdot w\ge\cos(\theta)\|v\|\}.
\]
If $V$ is a linear subspace of $\R^{N}$, we define $C(V,\theta)$ as the union of $C(w,\theta)$, $w\in V \cap \mathbb{S}^{N-1}$.

\begin{lemma}\label{lem:disjoint-away-from-Cone}
Let $\mm$ be a measure in $\R^{N}$ such that $(\supp \mm, |\cdot|_{\R^{N}},\mm)$ is a Lipschitz differentiability space.
Then for $\mm$-almost every $x_{0}\in\mathbb{R}^{N}$ there is an $n$-dimensional
subspace $V_{x_{0}}$ such that for every $\theta\in(0,\frac{\pi}{2}]$ 
there is a $\varrho=\varrho(x_{0})>0$ satisfying:
\[
\left(\supp \mm\cap B_{r}(x_{0}) \right)\backslash(x_{0}+C(V_{x_{0}},\theta))=\emptyset, \quad \forall r \in (0, \varrho].
\]
\end{lemma}
\begin{proof}
Since the coordinate functions are Lipschitz functions, for $\mm$-almost every $x_{0}\in\supp \mm$ there is a unique linear map $de_{k}(x_0):\R^{n}\to\R$ such that
\begin{equation}\label{eq:ekdek}
e_{k}(x)=e_{k}(x_{0})+de_{k}(x_{0})\cdot(\phi(x)-\phi(x_{0}))+o(\|x-x_{0}\|)
\end{equation}
for all $x\in\supp \mm$.
Note that the assignment $f \mapsto df(x_{0})$ is linear whenever two functions are both differentiable at $x_0$. Thus if $x_0$ is a point of differentiability of all coordinate functions then there is a linear map $D_{x_0}:\R^{N}\to\R^{n}$ such that for $\ell=\sum_{k=1}^{N}\lambda_k  \, e_k \in \R^{N}$ it holds
\begin{equation}\label{eq:defDx0}
 D_{x_0}\ell=\sum_{k=1}^{N}\lambda_k \, de_{k}(x_{0}).
\end{equation}

Let $V_{x_{0}}$ be the set of all $v\in\mathbb{R}^{N}$ such that
$\ell(v)=0$ for all $\ell\in\ker D_{x_{0}}$. The claim follows since
for all $\ell \in \ker D_{x_{0}}$ with $\|\ell\|=1$ we must have 
\[
\ell(x-x_{0})=o(\|x-x_{0}\|).
\]

It is easy to see that $V_{x_{0}}$ is at most $n$-dimensional. By
uniqueness of $de_{k}(x_0)$ the case $\dim (V_{x_{0}})<n$ cannot happen, compare also with \cite[Lemma 2.1]{BateS} and \cite[Lemma 3.3]{Bate}.
\end{proof}

\begin{corollary}\label{cor:id-is-chart}
Under the same hypothesis and notations of Lemma \ref {lem:disjoint-away-from-Cone}, 
assume moreover that  for $\mm$-almost every $x_0 \in \supp \mm$ it holds $V_{x_0} = \R^N$. Then $(\supp\mm, |\cdot|_{\R^N}, \mm)$ is a Lipschitz differentiability space w.r.t. the chart $(\supp \mm, \operatorname{id})$. 
\end{corollary}

\begin{proof}
By \eqref{eq:ekdek} and  by the very definition of $D_{x_{0}}$  given in \eqref{eq:defDx0}, we know that   for all $\ell \in \R^N$ and $\mm$-a.e.  $x_0 \in \supp \mm$ it holds
\begin{equation}\label{eq:ellphi}
\ell \cdot (x-x_0) = D_{x_0} \ell \cdot (\phi(x)-\phi(x_0)) + o(\|x-x_0\|).
\end{equation}
Assume $f$ is a Lipschitz function that is differentiable at $x_0$ w.r.t. to the chart $\phi$ and denote the $\phi$-relative Lipschitz differential by $d^\phi f (x_0)$. 
From the proof of the previous lemma we know that $D_{x_0}:\R^N\to\R^N$ is invertible. Set $df(x_0) := (D_{x_0})^{-1} d^\phi f (x_0)$, by applying \eqref{eq:ellphi} with $\ell= df(x_0)$ we get
\begin{eqnarray}
f(x) - f(x_0) &=& d^\phi f(x_0) \cdot (\phi(x)-\phi(x_0)) + o(\|x-x_0\|)  \nonumber \\
              &=& df(x_0) \cdot (x - x_0) + o(\|x-x_0\|)
\end{eqnarray}
which shows that the Lipschitz differential w.r.t. $(\supp \mm, \operatorname{id})$ exists at $x_{0}$. 
\\We now  show uniqueness  of the differential. The definition of Lipschitz differential implies that 
\[
  f (x_0) = \Lip\left( d^\phi f (x_0) \cdot (\phi(\cdot)-\phi(x_0) \right) (x_0)
\]
whenever $f$ is differentiable at $x_0$. Furthermore, equivalence shown in \cite[Lemma 2.1]{BateS} implies that for all $v \in \R^N$ and $\mm$-almost all $x_0\in \supp \mm$ it holds
\begin{equation}\label{eq:Lipv}
 \Lip \left( v \cdot (\phi(\cdot) -\phi(x_0)) \right) (x_0) > 0. 
\end{equation}
Since $D_{x_0}$ has trivial kernel for $\mm$-almost all $x_0 \in \supp \mm$, by using again \eqref{eq:ellphi} together with \eqref{eq:Lipv}  we see that 
\[
 \Lip \left( \ell \cdot (\cdot -x_0) \right) (x_0) = \Lip \left( D_{x_0} \ell \cdot (\phi(\cdot) -\phi(x_0)) \right) (x_0) > 0, \quad \forall \ell \in \R^{N}. 
\]
The uniqueness of the differential of $f$ at $x_{0}$ w.r.t.   $(\supp \mm, \operatorname{id})$ then follows by  \cite[Lemma 2.1]{BateS}, see also  \cite[Lemma 3.3]{Bate}.
\\We conclude that there is a subset  $\Omega \subset \supp \mm$ of full $\mm$-measure such that, for any $x_0 \in \Omega$, any Lipschitz function is differentiable at $x_0$ w.r.t. the chart  $\phi$ if and only if it is differentiable at $x_0$ w.r.t. to the identity chart $\operatorname{id}$. In particular, $(\supp\mm, |\cdot|_{R^N},\mm)$ is a Lipschitz differentiability space w.r.t.  the $N$-dimensional chart $(\supp \mm, \operatorname{id})$. 
\end{proof}

On the space $\mathcal{S}(\R^{N})$ of non-trivial linear subspaces of $\R^{N}$ we define a complete metric as follows
\[
d(V,W)=\inf \left\{\theta\in\left[0,\frac{\pi}{2}\right]\,\Big|\, W\subset C(V,\theta),V\subset C(W,\theta)\right\}
\]
for $V,W \in \mathcal{S}(\R^{N})$. In particular, for every $V,W\in\mathcal{S}(\R^{N})$ it holds
\[
d(V,W)\le\frac{\pi}{2}
\]
with equality for $W=V^{\bot}$.

Note that such a  metric $d$ can be equivalently defined as the Hausdorff metric of the $(N-1)$-sphere $S^{N-1}$ where each non-trivial linear subspace represents either a closed totally geodesic submanifold of $S^{N-1}$ or two antipodal points.

\begin{lemma}[{\cite[Lemma 15.13]{Mattila95}}]\label{lem:mattila}
Let $V\in\mathcal{S}(\R^{N})$ with $n=\dim (V)<N$ and let
$\theta\in(0,\frac{\pi}{2})$ and $r\in(0,\infty)$. Assume $E\subset\R^{N}$ is a subset satisfying  
\[
E\cap B_{r}(x)\cap(x+C(V^{\bot},\theta))=\{x\}, \quad \forall x \in E.
\]
Then, for any $x_{0}\in E$, the map $P_{V}:E\cap B_{\frac{r}{2}}(x_{0})\to V$
is bi-Lipschitz onto its image, where $P_{V}$ denotes the orthogonal projection
onto $V$. In particular, $E$ is $n$-rectifiable.
\end{lemma}
\begin{proof}
Fix $x_{0}\in E$ and let $x,y\in E\cap B_{\frac{r}{2}}(x_{0})$. The condition can be equally stated as  
\[
E\cap B_r(x) \subset x+C\left(V,\frac{\pi}{2}-\theta\right). 
\]
In particular, 
\[
(P_V(x-y))\cdot (x-y) \ge \cos\left(\frac{\pi}{2}-\theta\right)\|x-y\|^{2}.
\]
Since $x-y = P_V(x-y) + w$ for some $w\in V^{\bot}$, it holds
\[
\|P_V(x-y)\|^{2} = (P_V(x-y))\cdot(x-y).
\]
Therefore,
\[
\|P_{V}x-P_{V}y\|^{2}\ge\cos\left(\frac{\pi}{2}-\theta\right)\|x-y\|^{2}. 
\]
Setting $s=\cos(\frac{\pi}{2}-\theta)^{\frac{1}{2}}>0$ it follows that
\[
s\|x-y\|\le\|P_{V}x-P_{V}y\|\le\|x-y\|,
\]
i.e. $P_{V}:E\cap B_{\frac{r}{2}}(x_{0})\to V$ is bi-Lipschitz onto its image.
\end{proof}

\begin{proof}[Proof of Theorem \ref{thm:tangLDS}]
Since by Proposition \ref{prop:LocLDS} we know that   Borel subsets of positive measure in a Lipschitz differentiability space are still Lipschitz differentiability spaces with respect to the induced  metric measure structure, and since by assumption $(X,\sfd,\mm)$ is locally bi-Lipschitz embeddable in euclidean spaces, without loss of generality we can assume that all the space $(X,\sfd)$ is bi-Lipschitz embeddable into some $\R^{N}$ (otherwise just repeat the argument below for each $E_{j}\subset X$ which is by assumption bi-Lipschitz embeddable into some $\R^{N_{j}}$).
Since $\mm$ is an inner regular measure we can invade $U$ by an exhaustion of compact subsets, up to a negligible set. Composing with the bi-Lipschitz embedding into $\R^{N}$, we may even assume that $\mm$ is a finite measure supported on a compact set $K\subset \R^{N}$ . Using Lusin's and Egorov's Theorems we may further assume that the assignment $K\ni x\mapsto V_{x}\subset\R^{N}$ is continuous from the support of $\mm$ to the space  $\mathcal{S}(\R^{N})$ of linear subspaces on $\R^{N}$, where $V_{x}\subset \R^{N}$ is the $n$-dimensional linear subspace given by Lemma \ref{lem:disjoint-away-from-Cone}. 

The continuity of $x\mapsto V_{x}$ at $x_0$ implies that for all $\theta>0$ there is a $\delta=\delta(\theta)>0$ such that for all $x\in B_{\delta}(x_{0})$ 
\[
d(V_{x},V_{x_{0}})<\theta. 
\]
In particular, 
\[
d(V_{x},V_{x_{0}}^{\bot})>\frac{\pi}{2}-\theta
\]
and then by definition of the metric $d$ also
\[
V_{x_{0}}^{\bot}\cap C(V_{x},\theta)=\{\mathbf{0}\}.
\]
By Lemma \ref{lem:disjoint-away-from-Cone} for every $x\in K$ there is
an $r=r(\theta,x)\cap(0,1]$ such that 
\[
K\cap B_{r}(x)\backslash(x+C(V_{x},\theta))=\emptyset.
\]
Thus
\begin{equation}\label{eq:KCVx}
K\cap B_{r}(x)\cap(x+C(V_{x_{0}}^{\bot},\tilde{\theta}))=x
\end{equation}
for some fixed $\tilde{\theta}\in(0,\frac{\pi}{2}-\theta)$. Again
we can decompose $K$ into subsets $\{R_{j}\}_{j\in\mathbb{N}}$ such that
\eqref{eq:KCVx} holds for  $r=\frac{1}{j}$ whenever $x\in R_{j}$, that is  
\[
R_{j}\cap B_{\frac{1}{j}}(x)\cap(x+C(V_{x_{0}}^{\bot},\tilde{\theta}))=x.
\]
Lemma \ref{lem:mattila} finally implies that $R_{j}$ is bi-Lipschitz to a subset in $\R^{n}$, where $n=\dim (V_{x_{0}})$ is equal to the dimension of the original chart $(U,\phi)$. 
\end{proof}

In order to show that $\mm\llcorner R_{j}$ is absolutely continuous with respect to the relevant Hausdorff measure, we will make use of the next result proved by G. Alberti, M. Cs\"ornyei and D. Preiss \cite{ACP}  in two dimension and recently by G. De Philippis and F. Rindler \cite[Theorem 1.14]{DePRin} in higher dimensions.
\begin{theorem}\label{thm:DFR}
Let $\mm$  be a positive Radon measure on $\R^{d}$ such that every Lipschitz  function $f:\R^{d}\to \R$ is differentiable  $\mm$-almost everywhere.  Then $\mm$ is absolutely continuous with respect to the $d$-dimensional Lebesgue measure in $\R^{d}$.
\end{theorem}

Let us  mention that the aforementioned statement will also follow by a stronger result announced by  Cs\"ornyei and Jones \cite{CJ}, namely that for every Lebesgue null set $E\subset \R^{d}$ there exists a Lipschitz map $f:\R^{d}\to \R^{d}$ which is no-where differentiable; see also the discussion in the introduction of \cite{AM} for a detailed account of these type of results.  

If we combine Theorem \ref{thm:DFR} with the precise characterization of directions of non-differentiability by Alberti-Marchese \cite[Theorem 1.1]{AM} we get the following Corollary \ref{cor:abscont}.
Before stating it, let us recall the notion of decomposability bundle of a Radon measure $\mu$ introduced in \cite[Definition 2.6]{AM}.

Given a Radon measure $\mu$ over $\R^{d}$, we denote by ${\mathcal F}_{\mu}$ the class of all families $\{\mu_{t}: t \in I\}$ where $I$ is a measure space endowed with a probability measure $dt$ satisfying the following properties:
\begin{itemize}
\item [(a)] each  $\mu_{t}$ is the restriction of $\cH^{1}$ to a $1$-rectifiable set $E_{t}\subset \R^{d}$,
\item [(b)] the map $t \mapsto \mu_{t}(E)$ is measurable for every Borel subset $E \subset \R^{d}$ and $\int_{I} |\mu_{t}|(\R^{d}) \, dt<\infty$, where $|\mu_{t}|$ denotes the total variation measure associated to $\mu_{t}$,
\item [(c)] the measure $\int_{I} \mu_{t} \, dt$ is absolutely continuous with respect to $\mu$.
\end{itemize}
Then we denote by ${\mathcal G}_{\mu}$ the class of all Borel maps  $V: \R^{d} \to {\mathcal S}(\R^{d})$  such that for every $\{\mu_{t}: t \in I\} \in {\mathcal F}_{\mu}$ there holds
$$\Tan(E_{t}, x) \subset V(x), \quad \text{for $\mu_{t}$-a.e. $x$ and $dt$-a.e. $t \in I$}, $$
where $\Tan(E_{t}, x)$ denotes the tangent space of $E_{t}$ at $x$ which exists for $\cH^{1}$-a.e. $x \in E_{t}$ since by assumption $E_{t}$ is $1$-rectifiable.

Since ${\mathcal G}_{\mu}$ is closed under countable intersection (see \cite[Lemma 2.4]{AM}), it admits a $\mu$-minimal
element which is unique modulo equivalence $\mu$-a.e.. With a slight abuse of language
and notation we call any of these minimal elements \emph{decomposability bundle} of $\mu$,
and denote it by $x\mapsto V(x,\mu)$.
\\The motivation to consider the decomposability bundle is that, roughly said,  $V(x,\mu)$ represents the maximal vector space of differentiability of $\mu$ at $x$, see \cite[Theorem 1.1]{AM} for the precise statement.

\begin{corollary}\label{cor:abscont}
Let $\mu$ be a Radon measure on $\R^{d}$ and assume that $(\supp \mu, |\cdot|_{\R^{d}}, \mu)$ is a Lipschitz differentiability space with $d$-dimensional charts. Then every Lipschitz function $f:\R^{d}\to \R$ is differentiable $\mu$-almost everywhere. In particular, $\mu$ is absolutely continuous with respect to the Lebesgue measure on $\R^{d}$.
\end{corollary}

\begin{proof}

Arguing as in the beginning of the proof of Theorem \ref{thm:tangLDS}, we get that for $\mu$-almost every $x_0\in K=\supp \mm$ the space $V_{x_0}$ has dimension $d$ so that by Corollary \ref{cor:id-is-chart} the identity on $K$ is a $d$-dimensional chart. 
 
Assume $\mu$ is of the form
\[
\mu = \int \mu_\omega\, d\mathbb{P}(\omega)
\]
such that for $\mathbb{P}$-almost all $\omega$ the measure $\mu_\omega$ is supported on a  $1$-rectifiable curve $\gamma_\omega$. Such a disintegration-type formula is called \emph{Alberti representation} of $\mu$, see \cite[Definition 2.2]{Bate}. 
 Combining \cite[Theorem 1.1]{AM}  with  Theorem \ref{thm:DFR}, it follows that in order to get our thesis it is enough to show that 
\begin{equation}\label{eq:VxuRd}
V(x,\mu)=\R^d, \quad \text{ for }\mu\text{-a.e. } x,
\end{equation}
where $x\mapsto V(x,\mu)$ is the  decomposability bundle of $\mu$ defined above.
Indeed, if  \eqref{eq:VxuRd} holds, then every Lipschitz function is differentiable $\mu$-a.e. by \cite[Theorem 1.1(i)]{AM}, but then in view of Theorem \ref{thm:DFR} we get that $\mu$ is  absolutely continuous with respect to the Lebesgue measure on $\R^{d}$.

Assume by contradiction that \eqref{eq:VxuRd} is violated. Then restricting $K$, we may assume that $\mu(K)>0$ and there is a fixed cone $C\subset \R^{d}$ that is disjoint from $V(x,\mu)$ for all $x\in K$. Since by assumption $(K, |\cdot|_{\R^{d}}, \mu)$ is a Lipschitz differentiability space with $d$-dimensional charts, Bate \cite[Corollary 5.4]{Bate} showed that there is an Alberti representation in the direction of any cone $C'\subset \R^{d}$. Choosing  in particular $C'=C$, we get that for $\mu$-a.e. $x\in K$ there is a curve $\gamma_x$ that is differentiable at $\gamma^{-1}_x(x)$ and is tangent to $C$.  

However, by \cite[Theorem 1.1(ii)]{AM} there exists a Lipschitz function that is not differentiable in any direction $v\notin V(x,\mu)$ for $\mu$-a.e. $x \in K$. In particular, such a function is not differentiable along the curves $\gamma_x$ for $\mu$-a.e. $x \in K$. Since $\mu(K)>0$,  this contradicts the assumption that $(K, |\cdot|_{\R^{d}}, \mu)$  is a Lipschitz differentiability space. Thus \eqref{eq:VxuRd} holds and the proof is complete.
\end{proof}

Finally, recall that an $n$-rectifiable measure in $\R^N$ that is absolutely continuous with respect to the $n$-dimensional Hausdorff measure has almost everywhere unique linear tangents that are $n$-dimensional linear subspaces of $\R^N$, see \cite[Theorem 15.19]{Mattila95}. Thus we may summarize the content of the section in the next statement. 

\begin{theorem}\label{cor:tangLDS}
Let  $(X,\sfd,\mm)$ be a Lipschitz differentiability space and assume it is locally bi-Lipschitz embeddable in euclidean spaces, with charts $F_{k}:E_{k}\to \R^{N_{k}}$ bi-Lipschitz onto their image and $\mm(X \setminus \bigcup_{k \in \N} E_{k})=0$. Then there exists a countable collection $\{R_{j}\}_{j \in \N}$ of $\mm$-measurable subsets of $X$, covering $X$ up  to an $\mm$-negligible set, such that each $R_{j}$ is bi-Lipschitz to a measurable subset of $\R^{n_j}$ and $\mm \llcorner R_j$ is absolutely continuous with respect to the $n_j$-dimensional Hausdorff measure $\cH^{n_{j}}$.  In other words, $X$ is rectifiable as \emph{metric measure space}.
In addition, for $\mm$-almost every $x \in R_j \cap E_{k}$, the set $F_{k}(E_{k})$ has a  unique tangent at $F_{k}(x)$ that is an $n_j$-dimensional linear subspace of $\R^{N_{k}}$.
\end{theorem}

\begin{proof}
The first part of the statement corresponds to Theorem \ref{thm:tangLDS}; denote by $\Phi_{j}:R_{j}\to \R^{n_{j}}$ such bi-Lipschitz embeddings.  In order to show that  $\mm \llcorner R_j$ is absolutely continuous with respect to $\cH^{k_{j}}$, we first observe that  the property of being a Lipschitz differentiability space is invariant under composition with bi-Lipschitz maps, so that $(\supp ({\Phi_{j}}_{\sharp} (\mm \llcorner R_{j})), |\cdot|_{\R^{n_{j}}}, {\Phi_{j}}_{\sharp} (\mm \llcorner R_{j}))$ is a Lipschitz differentiability space with charts in $\R^{n_{j}}$. At  this point   we apply Corollary \ref{cor:abscont} and infer that  ${\Phi_{j}}_{\sharp} (\mm \llcorner R_{j})$  is absolutely continuous with respect to the  Lebesgue measure on $\R^{n_{j}}$ and therefore $\mm \llcorner R_{j}$ is absolutely continuous with respect to the $n_j$-dimensional Hausdorff measure $\cH^{n_{j}}$, since $\Phi_{j}$ is bi-Lipschitz. The uniqueness of tangent spaces follows then by \cite[Theorem 15.19]{Mattila95}.
\end{proof}

\section{Proof of Theorem \ref{thm:main}}
Let $(X,\sfd,\mm)$ be an $\RCD^{*}(K,N)$-space for some $K\in \R$ and $N\in (1,\infty)$. From Theorem \ref{thm:RCDLDS} we know that $(X,\sfd,\mm)$ is a Lipschitz differentiability space, namely there are countably many charts $\{(U_{i},\phi_{i})\}_{i\in \N}$ with  $\mm(U_{i})>0$,  $X=\cup_{i \in N} U_{i}$  and  $\phi_{i}:U_{i} \to \R^{n_{i}}$ Lipschitz maps with respect to which any Lipschitz function is differentiable almost
everywhere (see Definition \ref{def:LDS} for more details).

On the other hand, from Theorem \ref{RCD-rect} we know that there exists a countable collection $\{R_{j}\}_{j \in \N}$ of $\mm$-measurable subsets of $X$, covering $X$ up to an $\mm$-negligible set, such that each $R_{j}$ is bi-Lipschitz to a measurable subset of $\R^{k_{j}}$, for some $k_{j}\in \N$. A natural question is what is the relation between the dimensions $n_{i}$ and the  dimensions $k_{j}$, the first ones given by the Lipschitz differentiable structure and the second ones given by the rectifiability result. As observed by F. Cavalletti and T. Rajala \cite{CR}, the two agree  $\mm$-almost everywhere. More precisely the following  proposition holds true. 
\begin{proposition}\label{prop:ni=kj}
Let $(X,\sfd,\mm)$ be an $\RCD^{*}(K,N)$-space for some $K\in \R$ and $N\in (1,\infty)$ and let 
\begin{enumerate}
\item $\{(U_{i},\phi_{i})\}_{i\in \N}$, $\phi_{i}:U_{i} \to \R^{n_{i}}$ Lipschitz maps,  be the Lipschitz differentiable structure given by Theorem \ref{thm:RCDLDS} and Definition \ref{def:LDS},
\item $\{R_{j}\}_{j \in \N}$ collection of  $\mm$-measurable subsets of $X$, covering $X$ up to an $\mm$-negligible set, such that each $R_{j}$ is bi-Lipschitz to a measurable subset of $\R^{k_{j}}$ and for $\mm$-a.e. $x \in X$ the tangent space is unique and isometric to $\R^{k_{j}}$; i.e. the rectifiable structure given by Theorem \ref{RCD-rect}.
\end{enumerate}
If $\mm(U_{i}\cap R_{j})>0$, then  it holds $n_{i}=k_{j}$ and $\mm \llcorner (U_{i}\cap R_{j})$ is absolutely continuous with respect to the $k_{j}$-dimensional Hausdorff measure.
\end{proposition}

\begin{proof}
If $\mm(U_{i}\cap R_{j})>0$, by Proposition \ref{prop:LocLDS}, we know that $(U_{i}\cap R_{j}, \sfd|_{U_{i}\cap R_{j}\times U_{i}\cap R_{j}}, \mm\llcorner U_{i}\cap R_{j})$ is a Lipschitz differentiability space admitting a bi-Lipschitz embedding $\Phi_{j}$ into $\R^{k_{j}}$. Therefore, Theorem \ref{cor:tangLDS}  implies that  $U_{i}\cap R_{j}$ admits a unique tangent space at $\mm$-a.e.  $x \in U_{i}\cap R_{j}$ which is isometric to $\R^{n_{i}}$, $n_{i}$ being the dimension of the Lipschitz chart $\phi_{i}$. But,  on the other hand by (2),  for $\mm$-a.e. $x \in U_{i}\cap R_{j}$ the tangent space is unique and isometric to $\R^{k_{j}}$. It clearly  follows that $n_{i}=k_{j}$. 
\\Finally, again by Theorem   \ref{cor:tangLDS}, we know  that  $\mm\llcorner (U_{i}\cap R_{j})$ is absolutely continuous with respect to the $k_j$-dimensional Hausdorff measure $\cH^{k_{j}}$.
 \end{proof}

Since $(U_i)_{i\in\N}$ covers $X$ up to an $\mm$-negligible set and $\mm \llcorner (U_{i}\cap R_{j})$ is absolutely continuous with respect to the $k_j$-dimensional Hausdorff measure, we infer that the same holds for $\mm \llcorner R_j$ so that we can conclude the proof of Theorem \ref{thm:main}.

\begin{remark}[An alternative proof of Theorem \ref{thm:main}]
We decided to give a proof  of Theorem \ref{thm:main}  as self-contained as possibile but we wish to mention that by using more heavily the work of Bate \cite{Bate} and David \cite{David},  it is possibile to  give an alternative argument  which avoids the rectifiability Theorem \ref{cor:tangLDS}. Indeed the fact that  $n_{i}=k_{j}$ in Proposition \ref{prop:ni=kj} can be showed by using solely the uniqueness of tangent spaces to Lipschitz differentiability spaces embeddable in $\R^{d}$  proved in  \cite[Theorem 6.6]{Bate} with a different argument than ours.
Once we know that the dimension of the bi-Lipschitz charts given by Theorem  \ref{RCD-rect} agree with the dimension as Lipschitz differentiable space, we are reduced to prove the following statement:  if  $\mu$ is a positive finite measure supported on an compact subset of $\R^{n}$ such that $(\supp \mu, |\cdot|_{\R^{n}}, \mu)$ is a Lipschitz differentiability space with a unique Lipschitz chart with values in $\R^{n}$, then $\mu$ is absolutely continuous with respect to the Lebesgue measure of $\R^{n}$.
Such a statement can be showed as follows: by the work of Bate  \cite[Theorem 6.6]{Bate} the measure $\mu$ admits $n$-independent Alberti representantions.   But if a positive Radon measure in $\R^{n}$ admits $n$-independent Alberti representations then, by the recent work of De Philippis-Rindler \cite[Corollary 1.12]{DePRin},  it must be absolutely continuous with respect to the Lebesgue measure of $\R^{n}$, as desired.  
\end{remark}

\end{document}